\documentclass[10pt,a4paper]{article}
\usepackage{epsfig,amssymb,latexsym,amsthm}
\usepackage{amsmath}

\newtheorem{lemm}{Lemma}
\newtheorem{defi}{Definition}
\newtheorem{theo}{Theorem}
\newtheorem{rem}{Remark}
\newtheorem{cor}{Corollary}
\newtheorem{propo}{Proposition}

\newenvironment{definition}{\begin{defi}}{\end{defi}}
\newenvironment{lemma}{\begin{lemm}}{\end{lemm}}
\newenvironment{theorem}{\begin{theo}}{\end{theo}}
\newenvironment{remark}{\begin{rem}}{\end{rem}}

\newenvironment{prop}{\begin{propo}}{\end{propo}}

\newcommand{\bee}{\begin{equation}}
\newcommand{\ee}{\end{equation}}
\newcommand{\mb}{\mathbb}
\newcommand{\mf}{\mathbf}
\newcommand{\ml}{\mathcal}
\newcommand{\ds}{\displaystyle}
\newcommand{\bs}{\boldsymbol}

\begin{document}
\title{Optimal $C^{1,\frac12}$-regularity of $\ml H$-surfaces with a free boundary}
%\titlerunning{On the regularity of $H$-surfaces near smooth parts of the free boundary}
\author{Frank M\"uller\footnote{Frank M\"uller, Universit\"at Duisburg-Essen, Fakult\"at f\"ur Mathematik, 45117 Essen, e-mail: frank.mueller@uni-due.de}}
%\authorrunning{F.\,M\"uller}

\maketitle

\begin{abstract}
\noindent 
We consider continuous stationary surfaces of prescribed mean curvature in $\mb R^3$ -- short\-ly called $\ml H$-surfaces -- with part of their boundary varying on a smooth support manifold $S$ with non-empty boundary. We allow that the $\ml H$-surface meets the support manifold non-perpendicularly and presume the $\ml H$-surface to be continuous up to the boundary. Then we show: If $S$ belongs to $C^2$ resp.~$C^{2,\mu}$, then the $\ml H$-surface belongs to $C^{1,\alpha}$ for any $\alpha\in(0,\frac12)$ resp.~$C^{1,\frac12}$ up to the boundary. The latter conclusion is optimal by an example due to S.\,Hildebrandt and J.C.C.\,Nitsche. Our result extends a known theorem for the special case of minimal surfaces. In addition, we present asymptotic expansions at boundary branch points. 

\vspace{2ex}
\noindent Mathematics Subject Classification 2020: 53A10, 49N60, 49Q05, 35C20 
\end{abstract}

\vspace{2ex}
Let $S$ be a differentiable, two-dimensional manifold in $\mb R^3$ with boundary $\partial S$. Writing
	$$B^+:=\{w=(u,v)=u+iv\,:\ |w|<1,\ v>0\},\quad I:=(-1,1)\subset\partial B^+$$
for the upper unit half-disc in $\mb R^2\simeq\mb C$ and the straight part of its boundary, we consider \emph{surfaces of prescribed mean curvature} or shortly \emph{$\ml H$-surfaces} on $B^+$, i.e.~solutions of the problem
	\bee\label{g1.1}
	\begin{array}{l}
	\mf x\in C^2(B^+,\mb R^3)\cap C^0(\overline{B^+},\mb R^3)\cap H_2^1(B^+,\mb R^3),\\[1ex]
	\Delta\mf x=2\ml H(\mf x)\mf x_u\wedge\mf x_v\quad\mbox{in}\ B^+,\\[1ex]
	|\mf x_u|=|\mf x_v|,\quad\langle\mf x_u,\mf x_v\rangle=0\quad\mbox{in}\ B^+,
	\end{array}
	\ee
which satisfy the \emph{free boundary condition}
	\bee\label{g1.2}
	\mf x(I)\subset S\cup\partial S.
	\ee
Here $H_2^1(B^+,\mb R^3)$ denotes the Sobolev-space of measurable mappings $\mf x:B^+\to\mb R^3$, which are quadratically integrable together with their first derivatives. In addition, $\Delta=\frac{\partial^2}{\partial u^2}+\frac{\partial^2}{\partial v^2}$ stands for the Laplace operator in $\mb R^2$ and $\mf y\wedge\mf z$, $\langle\mf y,\mf z\rangle$ denote cross product and scalar product in $\mb R^3$, respectively; the latter notation will be used for vectors in $\mb C^3$, too. Finally, $\ml H\in C^0(\mb R^3,\mb R)$ is a precribed function. In (\ref{g1.1}), the system in the second line is called \emph{Rellich's system} and the third line contains the \emph{conformality relations}.

As is well-known, the restriction $\mf x|_{\ml R}$ of a solution of (\ref{g1.1}) to the set 
	$$\ml R:=\{w\in B^+\,:\ \nabla\mf x(w):=(\mf x_u(w),\mf x_v(w))\not=\mf 0\}$$ 
of \emph{regular points} parametrizes a surface with mean curvature $H=\ml H\circ \mf x$. We emphasize that singular points with $\nabla\mf x(w)=\mf 0$, so-called \emph{branch points}, are specifically allowed. This is natural from the viewpoint of the calculus of variations: If $\mf Q\in C^1(\mb R^3,\mb R^3)$ is a vector field with $\mbox{div}\,\mf Q=2\ml H$, then solutions of (\ref{g1.1}) appear as stationary points of the functional
	\bee\label{g1.3}
	E_{\mf Q}(\mf y):=\int\limits_{B^+}\Big\{\frac12|\nabla\mf y|^2+\langle\mf Q(\mf y),\mf y_u\wedge\mf y_v\rangle\Big\}\,du\,dv,
	\ee
where so-called \emph{inner} and \emph{outer variations} $\mf y$ of $\mf x$ are allowed. Roughly speaking, inner variation means a perturbation in the parameters $(u,v)$ and outer variations are perturbations in the space that retain the boundary condition (\ref{g1.2}); see \cite{dht} Section\,1.4 for the exact definitions in the minimal surface case $\mf Q\equiv\bs0$. For our purposes, it suffices to give the exact definition of outer variations:

\begin{definition}\label{d1}
Let $\mf x\in C^0(\overline{B^+},\mb R^3)\cap H_2^1(B^+,\mb R^3)$ fulfill the boundary condition (\ref{g1.2}). A perturbation $\mf x^{(\varepsilon)}(w):=\mf x(w)+\varepsilon\bs\phi(w,\varepsilon)$, $0\le\varepsilon \ll 1$, is called \emph{outer variation of $\mf x$}, if $\bs\phi(\cdot,\varepsilon)$ belongs to
	$$\ml A_{\mf x}:=\left\{\mf y\in H^1_2(B^+,\mb R^3)\,:\ \begin{array}{l} \mf y=\mf x\ \mbox{on}\ \partial B^+\setminus I\\[0.5ex] \mf y(w)\in S\ \mbox{for a.a.~}w\in I \end{array}\right\}$$
for any $\varepsilon$, if the family of Dirichlet's integrals 
	$$D\big(\bs\phi(\cdot,\varepsilon)\big):=\int\limits_{B^+}\Big(|\bs\phi_u(w,\varepsilon)|^2+|\bs\phi_v(w,\varepsilon)|^2\Big)\,du\,dv,\quad 0\le\varepsilon\ll 1,$$
is uniformly bounded in $\varepsilon$, and if $\bs\phi(\cdot,\varepsilon)\to\bs\phi(\cdot,0)\in H^1_2(B^+,\mb R^3)\ (\varepsilon\to 0+)$ holds true a.e.~on $B^+$. The function $\bs\phi_0:=\bs\phi(\cdot,0)$ is to be termed \emph{direction of the variation}.
\end{definition}

\begin{definition}\label{d2}
Let $\mf Q\in C^1(\mb R^3,\mb R^3)$ be given, define $E_{\mf Q}$ by formula (\ref{g1.3}) and set $\ml H:=\frac12\mbox{div}\,\mf Q$. A solution $\mf x:\overline{B^+}\to\mb R^3$ of (\ref{g1.1})--(\ref{g1.2}) is called \emph{stationary free $\ml H$-surface (w.r.t.~$E_{\mf Q}$)}, if we have
	$$\delta E_{\mf Q}(\mf x,\bs\phi_0):=\lim_{\varepsilon\to 0+}\frac1\varepsilon\big[E_{\mf Q}({\mf x}^{(\varepsilon)})-E_{\mf Q}(\mf x)\big]\ge0$$
for any outer variation $\mf x^{(\varepsilon)}=\mf x+\varepsilon\bs\phi(\cdot,\varepsilon)$, $0\le\varepsilon\ll1$. The quantity $\delta E_{\mf Q}(\mf x,\bs\phi_0)$ is called the \emph{first variation of $E_{\mf Q}$ at $\mf x$ in the direction $\bs\phi_0$}.    
\end{definition} 

Now we are able to formulate our main result:

\begin{theorem}\label{t1}
Let $S\subset\mb R^3$ be a differentiable two-manifold and assume a vector-field $\mf Q\in C^1(\mb R^3,\mb R^3)$ to be given such that 
	\bee\label{g1.4}
	|\langle\mf Q,\mf n\rangle|<1\quad\mbox{on}\ S\cup\partial S
	\ee
is satisfied; here $\mf n:S\cup\partial S\to\mb R^3$ denotes a unit normal field on $S$ which we locally extend continuously to $\partial S$. In addition, let $\mf x\in C^2(B^+,\mb R^3)\cap C^0(\overline{B^+},\mb R^3)\cap H_2^1(B^+,\mb R^3)$ be a stationary free $\ml H$-surface with $\ml H:=\frac12\mbox{div}\,\mf Q$. 
\begin{itemize}
\item[(i)]
If $S\in C^2$, then we have $\mf x\in C^{1,\alpha}(B^+\cup I,\mb R^3)$ for any $\alpha\in(0,\frac12)$.
\item[(ii)]
If $S\in C^{2,\beta}$ and $\mf Q\in C^{1,\beta}(\mb R^3,\mb R^3)$ for some $\beta\in(0,1)$, then we have $\mf x\in C^{1,\frac12}(B^+\cup I,\mb R^3)$.
\end{itemize}
\end{theorem}

\begin{remark}\label{r1}
For minimal surfaces, i.e.~the special case $\mf Q\equiv\bs0$, the result of Theorem\,\ref{t1} is due to R.\,Ye \cite{ye}. Under higher regularity assumptions on $S$ - namely $S\in C^3$ in case (i), $S\in C^4$ in case (ii) - these results for minimal surfaces were already proved by S.\,Hildebrandt and J.C.C.\,Nitsche \cite{hildebrandt-nitsche1}, \cite{hildebrandt-nitsche2}. In \cite{hildebrandt-nitsche2} the authors present an example showing the optimality of the regularity claimed in Theorem\,\ref{t1}\,(ii).
\end{remark}

\begin{remark}\label{r2}
In the minimal surface case, the assumption $\mf x\in C^0(\overline{B^+},\mb R^3)$ in Theorem\,\ref{t1} becomes redundant provided $S$ satisfies an additional uniformity condition. This is the famous continuity result for stationary minimal surfaces up to the free boundary, which is due to M.\,Gr\"uter, S.\,Hildebrandt, J.C.C.\,Nitsche \cite{grueter-hildebrandt-nitsche1}; see also G.\,Dziuk \cite{dziuk} regarding an analogue result for support surfaces without boundary. Concerning $\ml H$-surfaces, it is an open question whether stationarity implies continuity up to the boundary. However, there is an affirmative answer in the special case of vector-fields $\mf Q$ satisfying
	$$\langle\mf Q,\mf n\rangle=0\quad\mbox{on}\ S\cup\partial S;$$
see \cite{grueter-hildebrandt-nitsche2} for support surfaces without boundary, in \cite{mueller2} the case of support surfaces with boundary is shortly treated. In addition, minimality -- instead of the weaker assumption of stationarity -- implies continuity up to the boundary under very mild assumptions on $S$ and a smallness condition for $\mf Q$; see \cite{dht} Section\,2.5 or \cite{mueller3-} Section\,1.3.
\end{remark}

\begin{remark}\label{r3}
In the general case $\langle\mf Q,\mf n\rangle\not\equiv0$ on $S\cup\partial S$ the only results for stationary $\ml H$-surfaces known to the author are addressed to the case of support surfaces with empty boundary $\partial S=\emptyset$, see \cite{hildebrandt-jager}, \cite{hardt}, \cite{mueller3}.
\end{remark}

Our second theorem is concerned with boundary branch points:

\begin{theorem}\label{t2}
Let the assumptions of Theorem\,\ref{t1}\,(i) be satisfied and let $w_0\in I$ be a branch point of the stationary free $\ml H$-surface $\mf x$. If $\,\mf x:\overline{B^+}\to\mb R^3$ is non-constant, then there exist an integer $m\ge1$ and a vector $\mf a\in\mb C^3\setminus\{\bs0\}$ with $\langle\mf a,\mf a\rangle=0$, such that we have the representation
	\bee\label{g1.4+}
	\mf x_w(w)=\bs a(w-w_0)^m+o(|w-w_0|^m)\quad\mbox{as}\ w\to w_0.
	\ee
\end{theorem}

\begin{remark}\label{r4}
The proof of Theorem\,\ref{t2} can be found at the end of the paper; for branch points $w_0\in I$ with $\mf x(w_0)\in S$ the asymptotic expansion (\ref{g1.4+}) has been already proved in \cite{mueller3} Theorem\,1.13. The usual direct consequences as finiteness of boundary branch points in $\overline{B^+}\cap B_r(0)$ for any $r\in(0,1)$ and continuity of the surface normal of $\,\mf x$ up to the branch points follow; see e.g.~\cite{mueller3} Remarks 5.1 and 5.2.
\end{remark}

Preparing for the proof of Theorem\,\ref{t1}, we first have to localize the setting: Obviously, it suffices to show that for any $w_0\in I$ there exists some $\delta>0$ with $\mf x\in C^{1,\mu}(\overline{B_\delta^+(w_0)},\mb R^3)$ for $\mu\in(0,\frac12)$ or $\mu=\frac12$, respectively. Here we abbreviated 
	\begin{eqnarray*}
	&& B_\delta(w_0):=\{w=u+iv\in\mb C\,:\ |w-w_0|<\delta\},\\[1ex]
	&& B_\delta^+(w_0):=\{w=u+iv\in B_\delta(w_0)\,:\ v>0\}.
	\end{eqnarray*}
Since this result is included in Theorem\,1.3 of \cite{mueller3} for $w_0\in I$ with $\mf x_0:=\mf x(w_0)\in S$, we may assume $\mf x_0\in\partial S$. We localize around $\mf x_0$ which is possible according to the assumption $\mf x\in C^0(\overline{B^+},\mb R^3)$. After a suitable rotation and translation we can presume $\mf x_0=\bs0$ as well as the existence of some neighbourhood $\ml U=\ml U(\mf x_0)\subset\mb R^3$ and functions $\gamma\in C^2([-r,r])$, $\psi\in C^2(\overline{B_r(0)})$, $r>0$, with 
	\bee\label{g1.5}
	\gamma(0)=\frac d{ds}\gamma(0)=0,\quad\psi(0)=\nabla\psi(0)=0,
	\ee
such that we have the local representations
	\bee\label{g1.6}
	\begin{array}{l}
	S\cap\ml U=\big\{\mf p=(p^1,p^2,p^3)\in\Omega\times\mb R\,:\ p^3=\psi(p^1,p^2)\big\},\\[1ex]
	\partial S\cap\ml U=\big\{\mf p=(p^1,p^2,p^3)\in\Gamma\times\mb R\,:\ p^3=\psi(p^1,p^2)\big\},
	\end{array}
	\ee
where we abbreviated
	\bee\label{g1.7}
	\begin{array}{l}
	\Omega:=\big\{(p^1,p^2)\in B_r(0)\,:\ p^2>\gamma(p^1)\big\},\\[1ex]
	\Gamma:=\big\{(p^1,p^2)\in B_r(0)\,:\ p^2=\gamma(p^1)\big\}.
	\end{array}
	\ee
Now choose $\delta>0$ with $|\mf x(w)|<r$ for all $w\in\overline{B_\delta^+(w_0)}$. Since the system (\ref{g1.1}) is conformally invariant, we may reparametrize $\mf x|_{\overline{B_\delta^+(w_0)}}$ over $\overline{B^+}$ without renaming and obtain
	\bee\label{g1.8}
	\mf x(\overline{B^+})\subset\ml B_r:=\big\{\mf p\in\mb R^3\,:\ |\mf p|<r\big\},\qquad\mf x(0)=\bs0.
	\ee
In the following, we will repeatedly scale $r>0$ down -- sometimes without further command -- always assuming (\ref{g1.8}) to be satisfied.

Next we define
	\bee\label{g1.8-}
	q=q(\mf p):=Q^3(\mf p)-\psi_{p^1}(p^1,p^2)Q^1(\mf p)-\psi_{p^2}(p^1,p^2)Q^2(\mf p),
	\ee
where $Q^1,Q^2,Q^3$ are the components of $\mf Q$. Note that the smallness condition (\ref{g1.4}) and the normalization (\ref{g1.5}) imply $q\in C^1(\overline{\ml B_r})$ as well as
	\bee\label{g1.8+}
	|q(\mf p)|\le q_0<1\quad\mbox{for all}\ \mf p\in\overline{\ml B_r}
	\ee
with sufficiently small $r>0$; here $q_0\in(0,1)$ denotes some suitable constant. 

Writing $\dot\gamma:=\frac d{ds}\gamma$, we set
	\bee\label{g1.9}
	\begin{array}{l}
	z^1:=-i\psi_{p^1}x_w^1-i\psi_{p^2}x_w^2+ix_w^3,\\[1ex]
	z^2:=(1-i q\dot\gamma)x_w^1+(\dot\gamma+iq)x_w^2+(\psi_{p^1}+\psi_{p^2}\dot\gamma)x_w^3\quad\mbox{on}\ B^+.
	\end{array}
	\ee
Here we abbreviated $\psi_{p^j}=\psi_{p^j}(x^1,x^2)$, $\gamma=\gamma(x^1)$, and $q=q(\mf x)$, and we used one of the Wirtinger derivatives $x^j_w=\frac{\partial x^j}{\partial w}$ defined by the operators
	$$\frac{\partial}{\partial w}:=\frac12\Big(\frac{\partial}{\partial u}-i\frac{\partial}{\partial v}\Big),\quad\frac{\partial}{\partial\overline w}:=\frac12\Big(\frac{\partial}{\partial u}+i\frac{\partial}{\partial v}\Big).$$
As a first important observation we infer the following

\begin{prop}\label{p1}
The mapping $\mf z:=(z^1,z^2):B^+\to\mb R^3$ belongs to $C^1(B^+,\mb C^2)\cap L_2(B^+,\mb C^2)$ and satisfies the weak boundary condition
	\bee\label{g1.10}
	\liminf_{\varrho\to0}\bigg|\int\limits_{I_\varrho}\big\langle\bs\lambda(w),\mbox{Im}\,\mf z(w)\big\rangle\,du\bigg|=0\quad\mbox{for all}\ \bs\lambda\in C^1_c(B^+\cup I,\mb R^2),
	\ee
where we set $I_\varrho:=\{w=u+iv\in B^+\,:\ v=\varrho\}$ for $\varrho>0$. 
\end{prop}

\begin{proof}
The claimed regularity of $\mf z$ is obvious by definition. In order to prove (\ref{g1.10}), we set $\eta(s):=\psi(s,\gamma(s))$ and $\mf t(s):=(1,\dot\gamma(s),\dot\eta(s))$, $s\in(-r,r)$. Then $\mf t(s)$ is tangential to $\partial S$ at the point $(s,\gamma(s),\eta(s))$. If we choose $\alpha\in C^1_c(B^+\cup I)$ arbitrarily, the stationarity of $\mf x$ yields
	\bee\label{g1.11}
	\lim_{\varrho\to0+}\int\limits_{I_\varrho}\alpha\big\langle\mf t(x^1),\mf x_v+\mf Q(\mf x)\wedge\mf x_u\big\rangle\,du=0;
	\ee
this can be proved by combining the flow argument in \cite{dht} pp.~32--33 with \cite{mueller1} Lemma\,3. Now we set $\zeta:=\langle\mf t(x^1),\mf x_v+\mf Q(\mf x)\wedge\mf x_u\rangle$ and claim 
	\bee\label{g1.12}
	2\,\mbox{Im}\,z^2=-\zeta+(Q^2-\dot\gamma Q^1)(x_u^3-\psi_{p^1}x_u^1-\psi_{p^2}x_u^2)\quad\mbox{on}\ B^+,
	\ee
where we again abbreviated $Q^j=Q^j(\mf x)$, etc. Indeed, we compute
	\begin{eqnarray*}
	\zeta & = & x_v^1+Q^2x_u^3-Q^3x_u^2+\dot\gamma(x_v^2+Q^3x_u^1-Q^1x_u^3)+\dot\eta(x_v^3+Q^1x_u^2-Q^2x_u^1)\\[1ex]
	&=& x_v^1+\dot\gamma x_v^2-(Q^3-\psi_{p^1}Q^1-\psi_{p^2}Q^2)(x_u^2-\dot\gamma x_u^1)+(\psi_{p^1}+\psi_{p^2}\dot\gamma)x_v^3\\[1ex]
	&& +(Q^2-\dot\gamma Q^1)(x_u^3-\psi_{p^1}x_u^1-\psi_{p^2}x_u^2)\quad\mbox{on}\ B^+,
	\end{eqnarray*}
having $\dot\eta=\psi_{p^1}+\psi_{p^2}\dot\gamma$ in mind. Hence, the definition (\ref{g1.9}) of $z^2$ yields (\ref{g1.12}). 

Next we note the inequality
	\bee\label{g1.13}
	\int\limits_{I_\varrho}[x^3-\psi(x^1,x^2)]^2\,du\le c\varrho\int\limits_{B^+}|\nabla\mf x|^2\,du\,dv\le c\varrho,\quad \delta\in(0,1),
	\ee
with some constant $c>0$. This is an easy consequence of the boundary condition $x^3=\psi(x^1,x^2)$ on $I$ and the boundedness of $|\nabla\psi|$.

Now let $\bs\lambda=(\lambda_1,\lambda_2)\in C^1_c(B^+\cup I,\mb R^2)$ be chosen arbitrarily. Then we find
	\begin{eqnarray*}
	&&\hspace*{-5ex}\liminf_{\varrho\to0}\bigg|\int\limits_{I_\varrho}\big\langle\bs\lambda(w),\mbox{Im}\,\mf z(w)\big\rangle\,du\bigg|\\%[1ex]
	&& =\ \liminf_{\varrho\to0}\bigg|\int\limits_{I_\varrho}\big(\lambda_1\,\mbox{Im}\,z^1+\lambda_2\,\mbox{Im}\,z^2\big)\,du\bigg|^2\\%[1ex]
	&& \hspace*{-2.7ex}\stackrel{(\ref{g1.11}),(\ref{g1.12})}{=}\liminf_{\varrho\to0}\frac14\bigg|\int\limits_{I_\varrho}\big[\lambda_1+\lambda_2(Q^2-\dot\gamma Q^1)\big]\big[x_u^3-\psi_{p^1}x_u^1-\psi_{p^2}x_u^2\big]\,du\bigg|^2\\
	&& =\ \liminf_{\varrho\to0}\frac14\bigg|\int\limits_{I_\varrho}\big[x^3-\psi(x^1,x^2)\big]\frac\partial{\partial u}\big[\lambda_1+\lambda_2(Q^2-\dot\gamma Q^1)\big]\,du\bigg|^2\\%[1ex]
	\end{eqnarray*}
and are hence able to estimate
	\begin{eqnarray*}
	&&\hspace*{-5ex}\liminf_{\varrho\to0}\bigg|\int\limits_{I_\varrho}\big\langle\bs\lambda(w),\mbox{Im}\,\mf z(w)\big\rangle\,du\bigg|\\%[1ex]
	&& \le\ \liminf_{\varrho\to0}\frac14\int\limits_{I_\varrho}\big[x^3-\psi(x^1,x^2)\big]^2\,du\cdot\int\limits_{I_\varrho}\Big\{\frac\partial{\partial u}\big[\lambda_1+\lambda_2
	(Q^2-\dot\gamma Q^1)\big]\Big\}^2\,du\\%[1ex]
	&& \hspace*{-0.7ex}\stackrel{(\ref{g1.13})}{\le}\liminf_{\varrho\to0}c\varrho\bigg(1+\int\limits_{I_\varrho}|\nabla\mf x|^2\,du\bigg).
	\end{eqnarray*}
with an adjusted constant $c>0$. Using $\mf x\in H^1_2(B^+,\mb R^3)$, one can easily prove that the right hand side of this inequality vanishes (see e.g.~\cite{mueller3} Proposition\,2.1).
\end{proof}

In order to be able to relate the auxiliary function $\mf z$ with $\mf x$ we also need the following result:

\begin{prop}\label{p2}
The mapping $\mf z=(z^1,z^2)$ defined in (\ref{g1.9}) fulfils the relations
	\bee\label{g1.14}
	c^{-1}|\nabla\mf x|\le|\mf z|\le c|\nabla\mf x|\quad\mbox{on}\ B^+
	\ee
with some constant $c>0$.	
\end{prop}

\begin{proof}
The right-hand inequality in (\ref{g1.14}) is obvious by definition. In order to prove the left-hand inequality we write (\ref{g1.9}) as
	\bee\label{g1.15}
	\mf z=\mf A(\mf x)\cdot \begin{pmatrix} x_w^1\\[0.5ex] x_w^3 \end{pmatrix} +\mf b(\mf x) x_w^2\quad\mbox{on}\ B^+
	\ee
with
	\bee\label{g1.16}
	\mf A:=\begin{pmatrix} -i\psi_{p^1} & i \\[0.5ex] 1-iq\dot\gamma & \psi_{p^1}+\psi_{p^2}\dot\gamma \end{pmatrix},\quad \mf b:=\begin{pmatrix} -i\psi_{p^2} \\[0.5ex] \dot\gamma +iq \end{pmatrix}.
	\ee
Pick $0<\varepsilon<1-q_0$ arbitrarily. According to the normalization (\ref{g1.5}) we may choose $r=r(\varepsilon)>0$ sufficiently small to ensure 
	\bee\label{g1.17}
	|\det\mf A(\mf p)|\ge1-\varepsilon>0\quad\mbox{for}\ \mf p\in\overline{\ml B_r}.
	\ee
In particular, the inverse $\mf A^{-1}(\mf p)$ exists on $\overline{\ml B_r}$, and we conclude
	\bee\label{g1.18}
	\begin{pmatrix} x_w^1\\[0.5ex] x_w^3 \end{pmatrix}=\mf A^{-1}(\mf x)\cdot\mf z-\mf A^{-1}(\mf x)\cdot\mf b(\mf x)x_w^2\quad\mbox{on}\ B^+.
	\ee
Computing
	$$\mf A^{-1}\cdot\mf b=\frac1{\det\mf A}\begin{pmatrix} q-i[\psi_{p^1}\psi_{p^2}+(1+\psi_{p^2}^2)\dot\gamma]\\[0.5ex] 
	q(\psi_{p^1}+\psi_{p^2}\dot\gamma)+i(\psi_{p^2}-\psi_{p^1}\dot\gamma) \end{pmatrix},$$
the smallness (\ref{g1.8+}) of $q$, inequality (\ref{g1.17}), and the normalization (\ref{g1.5}) imply 
	$$|\mf A^{-1}(\mf p)\cdot\mf b(\mf p)|\le q_0+\varepsilon\quad\mbox{for}\ \mf p\in\overline{\ml B_r}$$
with sufficiently small $r=r(\varepsilon)>0$. Finally, we write the conformality relations in (\ref{g1.1}) as $\langle\mf x_w,\mf x_w\rangle=0$ in $B^+$, which yields
	$$|x^2_w|^2\le|x_w^1|^2+|x_w^3|^2\quad\mbox{on}\ B^+.$$
With these estimates we conclude
	$$\sqrt{|x_w^1|^2+|x_w^3|^2}\le c|\mf z|+(q_0+\varepsilon)\sqrt{|x_w^1|^2+|x_w^3|^2}\quad\mbox{on}\ B^+$$
from (\ref{g1.18}), where $c>0$ denotes a constant. Choosing e.g.~$\varepsilon=\frac{1-q_0}2$, we hence obtain the claimed estimate (\ref{g1.10}) with an aligned $c>0$.
\end{proof}

Combining Propositions \ref{p1} and \ref{p2}, we arrive at the following

\begin{lemma}\label{l1}
Let $\mf z=(z^1,z^2)$ be defined by (\ref{g1.9}). Set $B:=B_1(0)$, $B^-:=B\setminus(B^+\cup I)$, and consider the reflected function
	\bee\label{g1.19}
	\hat{\mf z}(w):=\left\{\begin{array}{ll}
	\mf z(w), & w\in B^+\\[1ex]
	\overline{\mf z(\overline w)}, & w\in B^-
	\end{array}\right.\in C^1(B\setminus I,\mb C^2)\cap L_2(B,\mb C^2).
	\ee
Then there exists $\mf h\in L_\infty(B,\mb C^2)$ such that $\hat{\mf z}$ solves the equation
	\bee\label{g1.20}
	\int\limits_B\big(\langle\hat{\mf z},\bs\varphi_{\overline w}\rangle+|\hat{\mf z}|^2\langle\mf h,\bs\varphi\rangle\big)\,du\,dv=0\quad\mbox{for all}\ \bs\varphi\in C^0_c(B,\mb C^2)\cap H_2^1(B,\mb C^2).
	\ee
\end{lemma}

\begin{proof}
The assertion follows from the estimate
	\bee\label{g1.21}
	|\hat{\mf z}_{\overline w}|\le c|\hat{\mf z}|^2\quad\mbox{on}\ B\setminus I,
	\ee
which we will prove below. Indeed, defining
	$$\mf h(w):=\left\{\begin{array}{ll}
	|\hat{\mf z}(w)|^{-2}\hat{\mf z}_{\overline w}, & \mbox{for}\ w\in B\setminus I\ \mbox{with}\ |\hat{\mf z}(w)|\not=0\\[1ex]
	0, & \mbox{otherwise}\end{array}\right.\in L_\infty(B,\mb C^2),$$
we infer $\hat{\mf z}_{\overline w}(w)=|\hat{\mf z}(w)|^2\mf h(w)$ away from isolated points in $B\setminus I$, because points $w\in B^+$ with $|\mf z(w)|=0$ are exactly the isolated branch points of $\mf x$. If we multiply this relation with an arbitrary $\bs\varphi\in C^1_c(B,\mb C^2)$, integrate over $B_{(\varrho)}^\pm:=\{w\in B^\pm\,:\ \pm v>\varrho\}$ and apply Gauss' integral theorem as well as the boundary condition, Proposition\,\ref{p1}, we arive at (\ref{g1.20}) for such $\bs\varphi$. By a standard approximation argument we can also allow $\bs\varphi\in C^0_c(B,\mb C^2)\cap H_2^1(B,\mb C^2)$ in (\ref{g1.20}).

In showing (\ref{g1.21}), the proof will be completed. To this end, we reflect $\mf x$ trivially across $I$,
	\bee\label{g1.22}
	\hat{\mf x}(w):=\left\{\begin{array}{ll}
	\mf x(w),& w\in B^+\cup I\\[1ex]
	\mf x(\overline w),& w\in B^-
	\end{array}\right..
	\ee
Defining $\mf A,\mf b\in C^1(\overline{\ml B_r})$ by (\ref{g1.16}) and having (\ref{g1.15}) in mind, we now may write $\hat{\mf z}$ as 
	\bee\label{g1.23}
	\hat{\mf z}=\mf A(\hat{\mf x})\cdot \begin{pmatrix} \hat x_w^1\\[0.5ex] \hat x_w^3 \end{pmatrix} +\mf b(\hat{\mf x})\,\hat x_w^2\quad\mbox{on}\ B^+
	\ee
and as
	\bee\label{g1.24}
	\hat{\mf z}=\overline{\mf A(\hat{\mf x})}\cdot \begin{pmatrix} \hat x_w^1\\[0.5ex] \hat x_w^3 \end{pmatrix} +\overline{\mf b(\hat{\mf x})}\,\hat x_w^2\quad\mbox{on}\ B^-.
	\ee
On the other hand, Rellich's system in (\ref{g1.1}) can be written as
	\bee\label{g1.25}
	\hat{\mf x}_{w\overline w}=\pm i\ml H(\hat{\mf x})\hat{\mf x}_{\overline w}\wedge\hat{\mf x}_w\quad\mbox{on}\ B^\pm.
	\ee
Differentiating (\ref{g1.23}), (\ref{g1.24}) and applying (\ref{g1.25}), we obtain
	$$|\hat{\mf z}_{\overline w}|\le c|\nabla\hat{\mf x}|^2\quad\mbox{on}\ B\setminus I$$
with some constant $c>0$. Hence, Proposition\,\ref{p2} yields the asserted relation (\ref{g1.21}).	
\end{proof}

Now the crucial step in the proof of Theorem\,\ref{t1} is the following

\begin{lemma}\label{l2}
For any $\mu\in(0,1)$, the mapping $\hat{\mf z}$ defined in Lemma\,\ref{l1} can be extended to a mapping of class $C^\mu(B,\mb C^2)$ with the property $\mbox{Im}\,\hat{\mf z}=\bs0$ on $I$.
\end{lemma}

\begin{proof}
We attempt to recover the steps in Section\,3 of \cite{mueller3}, which were used there to prove an analogue result, namely Lemma\,3.4.
\begin{enumerate}
\item
At first, we prove $\hat{\mf x}\in C^\beta(B,\mb R^3)$ for some $\beta\in(0,1)$. To this end, we consider the function
	\bee\label{g1.26}
	\chi:=\left\{\begin{array}{ll}
	\hat x^3-\psi(\hat x^1,\hat x^2) & \mbox{on}\ B^+\cup I\\[1ex]
	-\hat x^3+\psi(\hat x^1,\hat x^2) & \mbox{on}\ B^-
	\end{array}\right..
	\ee
Note that $\chi\in C^0(B)\cap H_2^1(B)$ is satisfied according to the boundary condition (\ref{g1.2}). Choose any disc $B_\varrho(w_0)\subset\subset B$ and define 
$\mf y=(y^1,y^2)\in C^\infty(B_\varrho(w_0),\mb R^2)\cap C^0(\overline{B_\varrho(w_0)},\mb R^2)$ as harmonic vector with boundary values
	$$y^1=\hat x^1,\quad y^2=\chi\quad\mbox{on}\ \partial B_\varrho(w_0).$$
Setting
	$$\bs\varphi:=\begin{pmatrix} -i(\chi-y^2) \\[0.5ex] \hat x^1-y^1 \end{pmatrix}\quad\mbox{on}\ \overline{B_\varrho(w_0)},\qquad\bs\varphi:=\mf0\quad\mbox{on}\ B\setminus\overline{B_\varrho(w_0)},$$
we obtain an admissible test function $\bs\varphi\in C_c^0(B,\mb C^2)\cap H_2^1(B,\mb C^2)$ for (\ref{g1.20}). We now insert $\bs\varphi$ and the relations (\ref{g1.23}), (\ref{g1.24}) for $\hat{\mf z}$ into (\ref{g1.20}) and use the special form (\ref{g1.16}) of $\mf A$ and $\mf b$. Writing $\bs\xi:=(\hat x^1,\hat x^3)$, we then find
	\begin{eqnarray*}
	&& \hspace*{-13ex}(1-d(r))\int\limits_{B_\varrho(w_0)}|\bs\xi_w|^2\,du\,dv \\
	&& \le\ (q_0+d(r))\int\limits_{B_\varrho(w_0)}|\bs\xi_w|\,|\hat x_w^2|\,du\,dv\\
	&& \hspace{3ex} +c\int\limits_{B_\varrho(w_0)}|\mf y_w|\,|\hat{\mf x}_w|\,du\,dv +\int\limits_{B_\varrho(w_0)}|\hat{\mf z}|^2|\mf h|\,|\bs\varphi|\,du\,dv,
	\end{eqnarray*}
where $c>0$ is a constant and $d(r)$, $0<r\ll 1$, denotes some (possibly varying) positive function satisfying $d(r)\to 0\,(r\to0+)$. By our general assumption (\ref{g1.8}), the maximum principle, and the normalization $\psi(0,0)=0$ we further get $|\bs\varphi|\le d(r)$. Using the conformality relations as well as Proposition\,\ref{p2} we hence conclude
	$$(1-q_0-d(r))\int\limits_{B_\varrho(w_0)}|\hat{\mf x}_w|^2\,du\,dv\le c\int\limits_{B_\varrho(w_0)}|\mf y_w|\,|\hat{\mf x}_w|\,du\,dv.$$
Applying the inequality of Cauchy-Schwarz and assuming $d(r)\le\frac12(1-q_0)$, we finally arrive at
	\bee\label{g1.26+}
	\int\limits_{B_\varrho(w_0)}|\nabla\hat{\mf x}|^2\,du\,dv\le c\int\limits_{B_\varrho(w_0)}|\nabla\mf y|^2\,du\,dv\quad\mbox{for all discs}\ B_\varrho(w_0)\subset\subset B.
	\ee
Note that there is a constant $c>0$ with
	$$c^{-1}|\nabla\hat{\mf x}|\le |\nabla(\hat x^1,\chi)| \le c |\nabla\hat{\mf x}|\quad\mbox{on}\ B,$$
due to the conformality relations and the condition $\nabla\psi(0,0)=0$. Employing C.\,B.\,Morrey's Dirichlet growth theorem, we hence infer $\hat{\mf x}\in C^\beta(B,\mb R^3)$ for some $\beta\in(0,1)$ from (\ref{g1.26+}).
\item
Next we show: For any $\alpha\in[0,2\beta)$ and any compact subset $K\subset B$ we have
	\bee\label{g1.27}
	\int\limits_B|w-w_0|^{-\alpha}|\hat{\mf z}(w)|^2\,du\,dv\le c\quad\mbox{for all}\ w_0\in K,
	\ee
where $c>0$ denotes a constant depending on $\alpha$ and $K$. 

We fix some $w_0\in K$ and define $\chi$ as in (\ref{g1.26}). We consider
	$$\bs\psi(w):=\begin{pmatrix} -i(\chi(w)-\chi(w_0))\\[0.5ex] \hat x^1(w)-\hat x^1(w_0) \end{pmatrix},\quad w\in B.$$
According to part\,1 of the proof we have $\chi,\hat x^1\in C^\beta(B)$ and conclude
	\bee\label{g1.28}
	|\bs\psi(w)|\le c|w-w_0|^\beta,\quad w\in K.
	\ee
Moreover, we can estimate (remember $\bs\xi=(\hat x^1,\hat x^3)$)
	\bee\label{g1.29}
	\begin{array}{rcl}
	\langle \hat{\mf z},\bs\psi_{\overline w}\rangle & \ge & |\bs\xi_w|^2-d(r)|\hat{\mf x}_w|^2-(q_0+d(r))|\bs\xi_w||\hat x_w^2|\\[1ex]
	& \ge & (1-q_0-d(r))|\bs\xi_w|^2\ \ge\ c(1-q_0-d(r))|\hat{\mf z}|^2\quad\mbox{in}\ B,
	\end{array}
	\ee
where we retained the notation of part\,1 and used Proposition\,\ref{p2}.

Now we choose some $\delta\in(0,\delta_0)$, $\delta_0:=\frac12\mbox{dist}(K,\partial B)$, and set
	$$\gamma(w):=\left\{\begin{array}{ll}
	\delta^{-\alpha}-\delta_0^{-\alpha},& 0\le|w-w_0|<\delta\\[1ex]
	|w-w_0|^{-\alpha}-\delta_0^{-\alpha},& \delta\le|w-w_0|<\delta_0\\[1ex]
	0,& \delta_0\le|w-w_0|\end{array}\right..$$
Then $\bs\phi:=\gamma\bs\psi\in C^0_c(B,\mb C^2)\cap H_2^1(B,\mb C^2)$ is admissible in (\ref{g1.20}) and relations (\ref{g1.28}), (\ref{g1.29}) as well as $|\langle \mf h,\bs\psi\rangle|\le d(r)$ yield
	\bee\label{g1.30}
	c(1-q_0-d(r))\int\limits_B\gamma|\hat{\mf z}|^2\,du\,dv\le c\int\limits_{\delta<|w-w_0|<\delta_0}|w-w_0|^{-\alpha-1+\beta}|\hat{\mf z}|\,du\,dv.
	\ee
We assume $d(r)\le\frac12(1-q_0)$ and apply the inequalities
	$$\int\limits_B\gamma|\hat{\mf z}|^2\,du\,dv\ge\int\limits_{\delta<|w-w_0|<\delta_0}|w-w_0|^{-\alpha}|\hat{\mf z}|^2\,du\,dv-\delta_0^{-\alpha}\int\limits_B|\mf{\hat z}|^2\,du\,dv$$
and
	\begin{eqnarray*}
	&& \hspace*{-7ex}\int\limits_{\delta<|w-w_0|<\delta_0}\hspace{-3ex}|w-w_0|^{-\alpha-1+\beta}|\hat{\mf z}|\,du\,dv
	\ \le\  \frac{\varepsilon}2\int\limits_{\delta<|w-w_0|<\delta_0}\hspace{-3ex}|w-w_0|^{-\alpha}|\hat{\mf z}|^2\,du\,dv\\[1ex]
	&& \hspace*{31.5ex} +\frac1{2\varepsilon}\int\limits_{\delta<|w-w_0|<\delta_0}\hspace{-3ex}|w-w_0|^{-\alpha-2+2\beta}\,du\,dv
	\end{eqnarray*}
with sufficiently small $\varepsilon>0$ to (\ref{g1.30}). Having $\int_B|\mf{\hat z}|^2\,du\,dv<+\infty$ as well as $2\beta>\alpha$ in mind, we arrive at 
	$$\int\limits_{\delta<|w-w_0|<\delta_0}\hspace{-3ex}|w-w_0|^{-\alpha}|\hat{\mf z}|^2\,du\,dv\le c$$
with some constant $c>0$ which is independent of $w_0\in K$ and $\delta\in(0,\delta_0)$. For $\delta\to0+$ we obtain the asserted estimate (\ref{g1.27}).
\item
Finally, it turns out that (\ref{g1.27}) is valid for $\alpha=1$. This can be proved exactly as in \cite{mueller3} Proposition\,3.3 via an induction argument using the representation formula of Pompeiu and Vekua, namely
	\bee\label{g1.31}
	\hat{\mf z}(w)=\mf y(w)-\frac1\pi\int\limits_B\frac{|\hat{\mf z}(\zeta)|^2\mf h(\zeta)}{\zeta-w}\,d\xi\,d\eta,\quad w\in B;\quad\zeta=\xi+i\eta,
	\ee
with some holomorphic vector $\mf y:B\to\mb C^2$. Hence $\hat{\mf z}$ is locally bounded in $B$. By applying E.\,Schmidt's inequality (see e.g.~\cite{dht} pp.~219--221) to a local version of (\ref{g1.31}), we conclude $\hat{\mf z}\in C^\mu(B,\mb C^2)$ for any $\mu\in(0,1)$, as asserted. The property $\mbox{Im}(\hat{\mf z})=\bs0$ on $I$ is now an immediate consequence of Proposition\,\ref{p1}.

\vspace{-4.5ex}
\end{enumerate}
\end{proof}

As the last preliminaries towards the proof of Theorem\,\ref{t1} we need two further lemmata; the first one is due to E.\,Heinz, S.\,Hildebrandt, and J.C.C.\,Nitsche and we present it in a special appropriate form:

\begin{lemma}{\bf(Heinz--Hildebrandt--Nitsche)}\label{l3}
\begin{itemize}
\item[(a)]
Let $f\in C^0(B^+,\mb C)$ be given such that its square $f^2$ has a continuous extension to $B^+\cup I$. Then $f$ can be extended to a continuous function $f\in C^0(B^+\cup I,\mb C)$. 
\item[(b)]
Let $f\in C^0([-\varrho_0,\varrho_0],\mb C)$ be given with some $\varrho_0\in(0,1)$. Suppose that $\mbox{Re}(f)\cdot\mbox{Im}(f)=0$ on $[-\varrho_0,\varrho_0]$ is satisfied and that there exist numbers $c>0$, $\alpha\in(0,1]$ with
	\bee\label{g1.32}
	|f^2(u_1)-f^2(u_2)|\le c|u_1-u_2|^{2\alpha}\quad\mbox{for all}\ u_1,u_2\in[-\varrho_0,\varrho_0].
	\ee
Then we have $f\in C^{\alpha}([-\varrho_0,\varrho_0],\mb C)$.
\end{itemize}
\end{lemma}

\begin{proof}
We refer to the Lemmata 3 and 4 in \cite{dht} Section\,2.7.
\end{proof}

The second of the announced lemmata contains a regularity result for generalized analytic functions, which we may attribute to I.N.\,Vekua; for the sake of completeness, we give the proof of a local version needed here:

\begin{lemma}{\bf(Vekua)}\label{l4}
Let $z\in C^1(B^+,\mb C)\cap C^0(B^+\cup I,\mb C)$ be a solution of 
	\bee\label{g0.1}
	z_{\overline w}=g\quad\mbox{in}\ B^+,\qquad \mbox{Im}\,z=h\quad\mbox{on}\ [-\varrho_0,\varrho_0]
	\ee
for some $\varrho_0\in(0,1)$. Then there hold:
\begin{itemize}
\item[(a)]
If $g\in C^0(B^+\cup I,\mb C)$ and $h\in C^\alpha([-\varrho_0,\varrho_0])$ for some $\alpha\in(0,1)$, then we have $z\in C^\alpha(\overline{B_\varrho^+(0)},\mb C)$ for any $\varrho\in(0,\varrho_0)$.
\item[(b)]
If $g\in C^\alpha(B^+\cup I,\mb C)$ and $h\in C^{1,\alpha}([-\varrho,\varrho])$ for some $\alpha\in(0,1)$, then we have $z\in C^{1,\alpha}(\overline{B_\varrho^+(0)},\mb C)$ for any $\varrho\in(0,\varrho_0)$.
\end{itemize}
\end{lemma}

\begin{proof}
\begin{enumerate}
\item
We first prove assertion (a). Fix some $\varrho\in(0,\varrho_0)$ and choose a test function $\phi\in C^\infty_c(B)$ with $\phi=1$ in $\overline{B_\varrho(0)}$ and $\phi=0$ in $B\setminus B_{\frac{\varrho+\varrho_0}2}(0)$ as well as a simply connected domain $B^+_{\frac{\varrho+\varrho_0}2}(0)\subset G\subset B_{\varrho_0}^+(0)$ with $C^2$-boundary. Let $\sigma:B\to G$ be a conformal mapping. Then the function $\tilde z:=(\phi z)\circ\sigma\in C^1(B,\mb C)\cap C^0(\overline B,\mb C)$ solves a boundary value problem
	\bee\label{g0.2}
	\tilde z_{\overline w}=\tilde g\quad\mbox{on}\ B,\qquad \mbox{Im}\,\tilde z=\tilde h\quad\mbox{on}\ \partial B,
	\ee
where $\tilde g\in C^0(\overline B,\mb C)$, $\tilde h\in C^{\alpha}(\partial B)$ ist satisfied; here one has to use (\ref{g0.1}) as well as the well-known Kellogg-Warschawski theorem on the boundary behaviour of conformal mappings, see e.g.~\cite{pommerenke}. By subtracting a holomorphic function in $B$ with boundary values $\tilde h$ we may assume $\tilde h\equiv0$; note that this holomorphic function belongs to $C^{\alpha}(\overline B,\mb C)$ by a well-known result of I.\,I.\,Privalov. Now, any solution of (\ref{g0.2}) with $\tilde h\equiv0$ has the form
	\bee\label{g0.3}
	\tilde z(w)=-\frac1\pi\int\limits_B\frac{\tilde g(\zeta)}{\zeta-w}\,d\xi\,d\eta-\frac w\pi\overline{\int\limits_B\frac{\tilde g(\zeta)}{1-\overline w\zeta}\,d\xi\,d\eta}+ z_0,\quad w\in\overline B,
	\ee
with some constant $z_0\in\mb R$; see Theorem\,2 in \cite{sauvigny} Chap.~IX, \S\,4. Defining the \emph{Vekua-Operator}
	$$T[\tilde g](w):=-\frac1\pi\int\limits_B\frac{\tilde g(\zeta)}{\zeta-w}\,d\xi\,d\eta,\quad w\in\mb C,$$
we may rewrite (\ref{g0.3}) as
	$$\tilde z(w)=T[\tilde g](w)+\overline{T[\tilde g]\Big(\frac1{\overline w}\Big)}+z_0,\quad w\in\overline B.$$
Well-known estimates for the Vekua-operator (see \cite{vekua} Chap.~I, \S\,6) now show $\tilde z\in C^{\alpha}(\overline B)$ and hence $z\in C^{\alpha}(\overline{B_\varrho^+(0)},\mb C)$. This proves (a).
\item
For the proof of claim (b) we repeat the construction above and note that, by (a), the right hand sides in (\ref{g0.2}) satisfy $\tilde g\in C^\alpha(\overline B,\mb C)$, $\tilde h\in C^{1,\alpha}(\partial B)$. Subtracting a holomorphic function with boundary values $\tilde h$, which belongs to $C^{1,\alpha}(\overline B,\mb C)$ by Privalov's theorem, we may again assume $\tilde h\equiv0$. According to Theorem\,2 in \cite{sauvigny} Chap.~IX, \S\,4 (see also \cite{vekua} Chap.~I, \S\,8) the solution (\ref{g0.3}) of this problem belongs to $C^{1,\alpha}(\overline B,\mb C)$ and we conclude $z\in C^{1,\alpha}(\overline{B_\varrho^+(0)},\mb C)$, as asserted.
\end{enumerate}
\end{proof}

We are now prepared to give the proof of our main result, Theorem\,\ref{t1}. To this end, we define a further auxiliary function, namely
	\bee\label{g1.33}
	z^3:=-(\dot\gamma+iq)x_w^1+(1-iq\dot\gamma)x_w^2+(\psi_{p^2}-\psi_{p^1}\dot\gamma)x_w^3\in C^1(B^+,\mb C)\cap H_2^1(B^+,\mb C)
	\ee
with $q=q(\mf x)$, $\dot\gamma=\dot\gamma(x^1)$, $\psi_{p^j}=\psi_{p^j}(x^1,x^2)$; remember the definitions of $\psi$, $\gamma$, and $q$ in (\ref{g1.6}), (\ref{g1.7}), and (\ref{g1.8-}). If we set $\bs\zeta:=(\mf z,z^3)=(z^1,z^2,z^3):B^+\to\mb C^3$, we have the identity
	\bee\label{g1.33+}
	\bs\zeta(w)=\mf B(\mf x(w))\cdot\mf x_w(w),\quad w\in B^+,
	\ee
where we abbreviated
	\bee\label{g1.34}
	\mf B:=\begin{pmatrix} -i\psi_{p^1} & -i\psi_{p^2} & i\\[1ex] 1-iq\dot\gamma & \dot\gamma+iq & \psi_{p^1}+\psi_{p^2}\dot\gamma\\[1ex] -(\dot\gamma+iq) & 1-iq\dot\gamma & \psi_{p^2}-\psi_{p^1}\dot\gamma \end{pmatrix}
	\in C^1(\overline{\ml B_r},\mb C^{3\times 3});
	\ee
see (\ref{g1.9}) for the definition of $\mf z=(z^1,z^2)$. Note that 
	$$\det\mf B=i(1+\dot\gamma^2)(1-q^2+|\nabla\psi|^2)\not=0\quad\mbox{on}\ \overline{\ml B_r}$$
is true according to the smallness condition (\ref{g1.8+}). Hence, the inverse $\mf B^{-1}(\mf p)$ exists for any $\mf p\in\overline{\ml B_r}$ and we have $\mf B^{-1}\in C^1(\overline{\ml B_r},\mb C^{3\times 3})$.

We intend to employ the conformality relations, which now can be written as
	\bee\label{g1.35}
	0=\langle\mf x_w,\mf x_w\rangle=\big\langle\mf B^{-1}(\mf x)\bs\zeta,\mf B^{-1}(\mf x)\bs\zeta\big\rangle=\langle\bs\zeta,\mf C(\mf x)\bs\zeta\rangle\quad\mbox{on}\ B^+
	\ee
with the matrix  $\mf C=(c_{ij})_{i,j=1,2,3}:=\mf B^{-T}\cdot\mf B^{-1}\in C^1(\overline{\ml B_r},\mb C^{3\times3})$. A lengthy but straightforward computation yields
	\bee\label{g1.36}
	\begin{array}{rcl}
	c_{11} &=& \ds-\frac{1-q^2}{1-q^2+|\nabla\psi|^2},\\[2.5ex]
	c_{12} &=& \ds\frac{q(\psi_{p^2}-\psi_{p^1}\dot\gamma)}{(1+\dot\gamma^2)(1-q^2+|\nabla\psi|^2)}\ =\ c_{21},\\[2.5ex]
	c_{13} &=& \ds-\frac{q(\psi_{p^1}+\psi_{p^2}\dot\gamma)}{(1+\dot\gamma^2)(1-q^2+|\nabla\psi|^2)}\ =\ c_{31},\\[2.5ex]
	c_{22} &=& \ds\frac{1+\dot\gamma^2+(\psi_{p^2}-\psi_{p^1}\dot\gamma)^2}{(1+\dot\gamma^2)^2(1-q^2+|\nabla\psi|^2)},\\[2.5ex]
	c_{23} &=& \ds-\frac{(\psi_{p^1}+\psi_{p^2}\dot\gamma)(\psi_{p^2}-\psi_{p^1}\dot\gamma)}{(1+\dot\gamma^2)^2(1-q^2+|\nabla\psi|^2)}\ =\ c_{32},\\[2.5ex]
	c_{33} &=& \ds\frac{1+\dot\gamma^2+(\psi_{p^1}+\psi_{p^2}\dot\gamma)^2}{(1+\dot\gamma^2)^2(1-q^2+|\nabla\psi|^2)}.
	\end{array}
	\ee
In particular, we have $\mf C:\overline{\ml B_r}\to{\mb R}^{3\times3}$. We are now ready to give the

\begin{proof}[Proof of Theorem\,\ref{t1}.]
\begin{enumerate}
\item
We write (\ref{g1.35}) in the form
	$$0=\sum_{j,k=1}^3c_{jk}z^jz^k=c_{33}(z^3)^2+2(c_{13}z^1+c_{23}z^2)z^3+\sum_{j,k=1}^2c_{jk}z^jz^k\quad\mbox{on}\ B^+,$$
where we abbreviated $c_{jk}=c_{jk}\circ\mf x$. Since $c_{33}>0$ holds on $\overline{\ml B_r}$ due to (\ref{g1.36}), we may rewrite this identity as
	\bee\label{g1.37}
	\Big(z^3+\sum_{j=1}^2\frac{c_{j3}}{c_{33}}z^j\Big)^2=\Big(\sum_{j=1}^2\frac{c_{j3}}{c_{33}}z^j\Big)^2-\sum_{j,k=1}^2\frac{c_{jk}}{c_{33}}z^jz^k\quad\mbox{on}\ B^+.
	\ee
With Lemma\,\ref{l2}, we extend the right hand side of (\ref{g1.37}) to a continuous function on $B^+\cup I$. Lemma\,\ref{l3}\,(a) thus yields that also $z^3+\sum_{j=1}^2\frac{c_{j3}}{c_{33}}z^j$ and, again due to Lemma\,\ref{l2}, $\bs\zeta=(z^1,z^2,z^3)$ can be extended continuously to $B^+\cup I$. The definition (\ref{g1.33+}) of $\bs\zeta$ as well as $\det\mf B\not=0$ then imply $\mf x\in C^1(B^+\cup I,\mb R^3)$.
\item
Now we prove part (i) of the theorem. For fixed $\varrho_0\in(0,1)$ and any $\mu\in(0,1)$ the right hand side of (\ref{g1.37}) belongs to $C^\mu([-\varrho_0,\varrho_0],\mb C)$ according to Lemma\,\ref{l2} and $\mf x\in C^1(B^+\cup I,\mb R^3)$. In addition, the imaginary part of the right hand side vanishes on $[-\varrho_0,\varrho_0]$ due to $\mbox{Im}(z^1)=\mbox{Im}(z^2)=0$ on $I$ (see again Lemma\,\ref{l2}) and to $\mf C:\overline{\ml B_r}\to{\mb R}^{3\times3}$ as shown above. Hence, the function $f=z^3+\sum_{j=1}^2\frac{c_{j3}}{c_{33}}z^j\in C^0(I,\mb C)$ satisfies the assumptions of Lemma\,\ref{l3}\,(b) for any $\alpha\in(0,\frac12)$. We conclude $f\in C^{\alpha}([-\varrho_0,\varrho_0],\mb C)$ and, by Lemma\,\ref{l2}, also $\bs\zeta\in C^{\alpha}([-\varrho_0,\varrho_0],\mb C^3)$ for any $\alpha\in(0,\frac12)$. If we differentiate (\ref{g1.33+}) w.r.t.~$\overline w$ and apply Rellich's system (\ref{g1.25}) we obtain 
	$$\bs\zeta_{\overline w}=\mf g\quad\mbox{on}\ B^+\qquad\mbox{with some}\quad \mf g\in C^0(B^+\cup I,\mb C^3).$$
Consequently, we may apply Lemma\,\ref{l4}\,(a) to $\bs\zeta$ and find $\bs\zeta\in C^\alpha(\overline{B_\varrho^+(0)},\mb C^3)$ as well as $\mf x\in C^{1,\alpha}(\overline{B_\varrho^+(0)},\mb R^3)$ for any $\varrho\in(0,\varrho_0)$ and any $\alpha\in(0,\frac12)$. Since we localized around an arbitrary point $w_0\in I$, the proof of Theorem\,\ref{t1}\,(a) is completed.
\item
For the proof of Theorem\,\ref{t1}\,(ii) we assume $\ml S\in C^{2,\beta}$, $\mf Q\in C^{1,\beta}(\mb R^3,\mb R^3)$ with some $\beta\in(0,1)$. Then we also have $\mf B\in C^{1,\beta}(\overline{\ml B_r},\mb R^{3\times3})$ and by part (i) we know, for instance,  $\mf x\in C^{1,\frac14}(B^+,\mb R^3)$. Set $\gamma:=\min\{\frac14,\beta\}$, define $\mf z=(z^1,z^2)$ by (\ref{g1.9}) and differentiate these equations w.r.t.~$\overline w$. Then we obtain 
	$$\mf z_{\overline w}=\mf g_0\quad\mbox{on}\ B^+,\qquad \mbox{Im}\,\mf z=\bs0\quad\mbox{on}\ I$$
with some $\mf g_0\in C^\gamma(B^+\cup I,\mb C^2)$. From Lemma\,\ref{l4}\,(b) we thus conclude $\mf z\in C^{1,\gamma}([-\varrho,\varrho],\mb C^2)$ for any $\varrho\in(0,1)$. In particular, the right hand side of equation (\ref{g1.37}) belongs to $C^{1}([-\varrho,\varrho],\mb C)$ and Lemma\,\ref{l3}\,(b) shows $\bs\zeta\in C^{\frac12}([-\varrho,\varrho],\mb C^3)$ for any $\varrho\in(0,1)$. Now Lemma\,\ref{l4}\,(a) can be applied to get $\bs\zeta\in C^{\frac12}(\overline{B_\varrho^+(0)},\mb C^3)$ and we finally arrive at $\mf x\in C^{1,\frac12}(B^+\cup I,\mb R^3)$, as asserted.

\vspace{-2.5ex}
\end{enumerate}
\end{proof}

We conclude the paper with the

\begin{proof}[Proof of Theorem\,\ref{t2}.]
We choose a branch point $w_0\in I$ and assume $\mf x(w_0)\in\partial S$; compare Remark\,\ref{r4} above. We localize as above -- note especially $w_0\mapsto 0$ -- and define $\mf z=(z^1,z^2)$ by (\ref{g1.9}). Reflecting $\mf z$ as in (\ref{g1.19}), the resulting function $\hat{\mf z}:B\to\mb C^2$ satisfies $\hat{\mf z}\in C^1(B\setminus I,\mb C^2)\cap C^0(B,\mb C^2)$ and $\mbox{Im}\,\hat{\mf z}=\bs0$ on $I$ according to Lemma\,\ref{l2}.

Now choose an arbitrary domain $D\subset\subset B$ with piecewise smooth boundary. Then the arguments leading to formula (\ref{g1.20}) in Lemma\,\ref{l1} yield
	$$\frac1{2i}\oint\limits_{\partial D}\langle\hat{\mf z},\bs\varphi\rangle\,dw=\int\limits_D\big(\langle\hat{\mf z},\bs\varphi_{\overline w}\rangle+|\hat{\mf z}|^2\langle\mf h,\bs\varphi\rangle\big)\,du\,dv\quad\mbox{for all}
	\ \bs\varphi\in C^1(B,\mb C^2);$$
here $\mf h:B\to\mb C^2$ denotes some bounded function. According to the boundedness of $\hat{\mf z}$ on $D$ we find a constant $c>0$ such that
	$$\bigg|\oint\limits_{\partial D}\langle\hat{\mf z},\bs\varphi\rangle\,dw\bigg|\le 2\int\limits_D\big(|\bs\varphi_{\overline w}|+c|\bs\varphi|\big)|\hat{\mf z}|\,du\,dv\quad\mbox{for all}
	\ \bs\varphi\in C^1(B,\mb C^2).$$
The Hartman-Wintner technique -- see e.g.~Theorem\,1 in \cite{dht} Section 3.1 -- now implies the existence of some $m\in\mb N$ and some vector $\hat{\mf b}\in\mb C^2\setminus\{\bs0\}$ such that 
	\bee\label{g1.38}
	\hat{\mf z}(w)=\hat{\mf b}w^m+o(|w|^m)\quad\mbox{as}\ w\to 0.
	\ee
Note here that $\hat{\mf z}$ cannot vanish identically in $B$ since, otherwise, we would have $\nabla\mf x\equiv\bs0$ near $w_0$ due to Proposition\,\ref{p2}; this is impossible by our assumption $\mf x\not\equiv\mbox{const}$ as can be easily seen by employing the well-known asymptotic expansions at interior branch points.

Next, we define $z^3$ by (\ref{g1.33}) and consider $\bs\zeta=(z^1,z^2,z^3)=(\mf z,z^3)$, which can be extended to a continuous function on $\overline{B_\varrho^+(0)}$ for any $\varrho\in(0,1)$, according to part\,2 in the proof of Theorem\,\ref{t1}. In addition, we recall the relation (\ref{g1.37}), where the quantities $c_{jk}=c_{jk}\circ\mf x$ are continuous functions on $\overline{B^+}$.

Now we multiply (\ref{g1.37}) by $w^{-2m}$ and let $w\in\overline{B^+_\varrho(0)}$ tend to $0$. Due to (\ref{g1.38}), the right hand side and hence also the left hand side converge. Applying (\ref{g1.38}) again as well as a variant of Lemma\,\ref{l3}\,(a), we find $w^{-m}z^3(w)\to b^3$ as $w\to0$ with some limit $b^3\in \mb C$. Setting $\mf b:=(\hat{\mf b},b^3)\in\mb C^3$, we conclude
	\bee\label{g1.39}
	\bs\zeta(w)=\mf bw^m+o(|w|^m)\quad\mbox{as}\ w\to 0.
	\ee
This relation finally yields the announced expansion (\ref{g1.4+}) according to $\mf x_w=(\mf B^{-1}\circ\mf x)\bs\zeta$; see (\ref{g1.33+}) and recall $\det\mf B\not=0$. The relation $\langle\mf a,\mf a\rangle=0$ is now a direct consequence of the conformality relations and (\ref{g1.4+}).
\end{proof}

\end{document}